\documentclass[final]{dmtcs-episciences}

\usepackage[usenames]{color}
\usepackage{amssymb}
\usepackage{amsmath}
\usepackage{amsthm}
\usepackage{amsfonts}
\usepackage{amscd}
\usepackage{lscape}
\usepackage[title]{appendix}
\usepackage{slashbox}
\usepackage{color}
\usepackage{float}
\usepackage{graphics}
\usepackage{latexsym}
\usepackage{breakurl}

\DeclareMathOperator{\per}{per}

\def\modd#1 #2{#1\ \mbox{\rm (mod}\ #2\mbox{\rm )}}

\DeclareMathOperator{\cexp}{ce}
\DeclareMathOperator{\AO}{AO}
\DeclareMathOperator{\AN}{AN}

\newcommand{\br}{\overline{r}}
\newcommand{\bs}{\overline{s}}
\newcommand{\bt}{\overline{t}}
\newcommand{\bu}{\overline{u}}

\newcommand{\bx}{\overline{x}}

\usepackage[round]{natbib}

\author[Aseem Baranwal et al.]{Aseem Baranwal\affiliationmark{1}
  \and James Currie\affiliationmark{2}\thanks{The research of JC is supported by
  NSERC Grant 2017-03901.}
  \and Lucas Mol\affiliationmark{3}
  \and Pascal Ochem\affiliationmark{4}\\
  \and Narad Rampersad\affiliationmark{2}\thanks{The research of NR is supported by NSERC Grant 2019-04111.}
  \and Jeffrey Shallit\affiliationmark{1}\thanks{The research of JS is supported by NSERC Grant 2018-04118.}}
\title[Antisquares and Critical Exponents]{Antisquares and Critical Exponents}
\affiliation{
University of Waterloo, Canada\\
University of Winnipeg, Canada\\
Thompson Rivers University, Canada\\
LIRMM, CNRS, Universit\'e de Montpellier, France}
\begin{document}

\publicationdata{vol. 25:2 }{2023}{11}{10.46298/dmtcs.10063}{2022-09-20; 2022-09-20; 2023-07-27}{2023-08-02}

\theoremstyle{plain}
\newtheorem{theorem}{Theorem}
\newtheorem{corollary}[theorem]{Corollary}
\newtheorem{lemma}[theorem]{Lemma}
\newtheorem{proposition}[theorem]{Proposition}

\theoremstyle{definition}
\newtheorem{definition}[theorem]{Definition}
\newtheorem{example}[theorem]{Example}
\newtheorem{conjecture}[theorem]{Conjecture}
\newtheorem{problem}[theorem]{Problem}

\theoremstyle{remark}
\newtheorem{remark}[theorem]{Remark}

\maketitle

\begin{abstract}
The (bitwise) complement $\overline{x}$ of a binary word $x$ is obtained by changing each $0$ in $x$ to $1$ and vice versa. An {\it antisquare\/} is a nonempty word of the form $x\, \overline{x}$.  In this paper, we study infinite binary words that do not contain arbitrarily large antisquares.
For example, we show that the repetition threshold for the 
language of infinite binary words containing 
exactly two distinct antisquares is
$(5+\sqrt{5})/2$. We also study
repetition thresholds for related classes, where ``two'' in the previous sentence is replaced by a larger number.

We say a binary word is {\it good\/} if the only antisquares it contains are $01$ and $10$. We characterize the minimal antisquares, that is, those words that are antisquares but all proper factors are good.   We determine the growth rate of the number of good words of length $n$ and determine the repetition threshold between polynomial and exponential growth for the number of good words.
\end{abstract}

\section{Introduction}

Let $x$ be a finite nonempty binary word.  We say
that $x$ is an {\it antisquare} if there exists
a word $y$ such that $x = y \, \overline{y}$, where
the overline denotes a morphism that maps
$0 \rightarrow 1$ and $1 \rightarrow 0$.
For example, $011100$ is an antisquare.  The {\it order\/} of an antisquare $y\, \overline{y}$ 
is defined to be $|y|$, where $|y|$ denotes the length of $y$.

Avoidance of antisquares has been studied previously in combinatorics on words.  For example, \cite{Mousavi&Schaeffer&Shallit:2016}
proved that the infinite Fibonacci word
$${\bf f } = 01001010 \cdots,$$
the fixed point of the morphism
$0 \rightarrow 01$, $1 \rightarrow 0$, has
exactly four antisquare factors, namely,
 $01, 10, 1001$, and $10100101$.   More generally,
 all the antisquares in Sturmian words have
 recently been characterized in \cite{Hieronymi:2022}.
On the other hand, \cite{Ng&Ochem&Rampersad&Shallit:2019} classified those infinite binary words containing the
minimum possible numbers of distinct squares and antisquares.

It is easy to see that no infinite binary word,
except the trivial families given by
$(0+\epsilon) 1^\omega$ and $(1+\epsilon) 0^\omega$, can contain
at most one distinct antisquare. 
(Here the notation $x^\omega$ refers to the right-infinite word $xxx\cdots$.) However, once we move to two distinct antisquares, the situation is quite different.  We have the following:
\begin{proposition}
There are exponentially many finite binary words of length $n$ having at most two distinct antisquares, and there are uncountably many infinite binary words with the same property.
\end{proposition}

\begin{proof}
It is easy to see that every binary word in
$\{1000,10000\}^*$ has only the antisquares
$01$ and $10$, which proves the first claim.

For the second, consider the uncountable set of
infinite words $\{1000,10000\}^\omega$.  (Here, by
$S^\omega$ for a set $S$ of nonempty
finite words, we mean
the set of all infinite words arising
from concatenations of elements of $S$.)
\end{proof}

Furthermore, it is easy to see that if an infinite binary word, other than $001^{\omega}$ and its complement, contains exactly two antisquares, then these antisquares must be $01$ and $10$.
Call a binary word \emph{good} if it contains no antisquare factors, except possibly $01$ and $10$.
This suggests studying the following problem.
\begin{problem}
Find the repetition threshold for good words.
\label{prob1}
\end{problem}

The repetition threshold for a class
of (finite or infinite) words is defined as follows.  First, we
say that a finite word $w=w[1..n]$ has {\it period\/}
$p\geq 1$ if $w[i]=w[i+p]$ for $1 \leq i \leq n-p$.
The smallest period of a word $w$ is called
{\it the\/} period, and we write it as
$\per(w)$.  The {\it exponent\/} of a finite word
$w$, written $\exp(w)$ is defined to be
$|w|/\per(w)$.   We say a word
(finite or infinite) is
{\it $\alpha$-free\/} if the exponents of its nonempty factors are all $<\alpha$.   We say a word is {\it $\alpha^+$-free\/} if the exponents of its nonempty factors are all $\leq\alpha$.  The {\it critical exponent\/} of
a finite or infinite word $x$ is the supremum, over all
nonempty finite factors $w$ of $x$, of
$\exp(w)$; it is written $\cexp(x)$. Finally, the {\it repetition
threshold} for a language $L$ of infinite words is
defined to be the infimum, over all
$x \in L$, of $\cexp(x)$.

The critical exponent of a word can be either rational or irrational.   If it is rational, then it can either be attained by a particular finite factor, or not attained.  For example, the critical exponent of both
\begin{itemize}
    \item the Thue-Morse word
    ${\bf t} = 0110100110010110100101100 \cdots$,
    fixed point of the morphism $0 \rightarrow 01$, $1 \rightarrow 10$; and
    \item the variant Thue-Morse word
    ${\bf vtm} = 2102012101202102012021012\cdots$,
    fixed point of the morphism
    $2 \rightarrow 210$, $1 \rightarrow 20$, $0 \rightarrow 1$
\end{itemize}
is $2$, but it is attained in the former case and not attained in the latter.  If the critical exponent $\alpha$ is attained, we typically write it as $\alpha^+$.

In 1972, \cite{Dejean:1972} wrote a classic paper on combinatorics on words, where she determined the repetition threshold for the language of all infinite words over $\{ 0, 1, 2 \}$---it is $\tfrac74^+$---and conjectured the value of the repetition threshold for the languages $\Sigma_k^*$ for $k \geq 4$,
where $\Sigma_k = \{ 0,1,\ldots, k-1 \}$.
Dejean's conjecture was only completely resolved in 2011, in
\cite{Rao:2011} and \cite{Currie&Rampersad:2011}, independently.   

The repetition threshold has been studied for many classes of words.  To name a few, there are the
\begin{itemize}
    \item Sturmian words, studied in \cite[Prop.~15]{Carpi&deLuca:2000};
    \item palindromes, studied in \cite{Shallit:2016};
    \item rich words, studied in \cite{Currie&Mol&Rampersad:2020};
    \item balanced words, studied in \cite{Rampersad&Shallit&Vandomme:2019}
    \cite{Dvorakova&Opocenska&Pelantova&Shur:2022}; and
    \item complementary symmetric Rote words, studied in
    \cite{Dvorakova&Medkova&Pelantova:2020}.
\end{itemize}
For variations on and generalizations of repetition threshold, see 
\cite{Ilie&Ochem&Shallit:2005,Badkobeh&Crochemore:2011,Fiorenzi&Ochem&Vaslet:2011,Samsonov&Shur:2012,Mousavi&Shallit:2013}.

The goal of this paper is to study the repetition threshold for two classes of infinite words:
\begin{itemize}
    \item $\AO_\ell$, the binary words avoiding all antisquares of order $\geq \ell$;
    \item $\AN_n$, the binary words with no more than $n$ antisquares.
\end{itemize}
It turns out that there is an interesting and subtle hierarchy, depending on the values of $\ell$ and $n$.

Our work is very similar in flavor to that of
\cite{Shallit:2004}, which found a similar hierarchy concerning critical exponents and sizes of squares avoided.  The hierarchy for antisquares, as we will see, however, is significantly more complex.

In this paper, in Sections~\ref{two} and \ref{three}, we solve Problem~\ref{prob1}, and show
that the repetition threshold for good words is $2+\alpha$,
where $\alpha = (1+\sqrt{5})/2$ is the golden ratio.

Proving that the repetition threshold for a class of infinite words equals some real number $\beta$ generally consists of two parts:  first, an explicit construction of a word avoiding $\beta^+$ powers.  This is often carried out by finding an appropriate morphism $h$ whose infinite fixed point $\bf x$ (or an image of $\bf x$ under a second morphism) has the desired property.   Second,
if $\beta$ is rational, then one can prove there is no infinite word
avoiding $\beta$-powers by a breadth-first or depth-first search of the infinite tree of all words.   If $\beta$ is irrational, however, one must generally be more clever.

In Section~\ref{AO_l} we determine the repetition threshold for binary words avoiding all antisquares of order $\geq \ell$,
and in Section~\ref{AN_n} we determine the repetition threshold for binary words with no more than $n$ antisquares.
In Section~\ref{manti} we completely characterize the minimal antisquares; i.e., the binary words that are antisquares but have
the property that all proper factors are good.  This characterization is then used in Section~\ref{enum}, where
we determine the growth rate of the number of good words of length $n$.
In this section we also show that the repetition threshold between polynomial and exponential growth for good words avoiding $\alpha$-powers is $\alpha=\tfrac{15}4$; i.e., there are exponentially many
such words avoiding $\tfrac{15}4^+$-powers, but only polynomially many that avoid $\tfrac{15}4$-powers.

\section{A good infinite word with critical exponent \texorpdfstring{$2+\alpha$}{2+alpha}}
\label{two}

Consider the morphisms below:
\begin{center}
\begin{tabular}{ r l c r l }
 $\varphi$: & $0 \mapsto001$ & & $g$: & $0 \mapsto01$ \\
         & $1 \mapsto01$ & & & $1 \mapsto11$.
\end{tabular}
\end{center}
Let us write $\varphi^\omega(0)$ for the 
(unique) infinite fixed point of $\varphi$ that
starts with $0$.
We claim that the infinite word ${\bf w} = g(\varphi^\omega(0))$ does not have antisquares
other than $01$ and $10$, and has critical exponent $2 + \alpha$, where
$\alpha = (1+\sqrt{5})/2$.  The infinite word ${\bf w}$ is Fibonacci-automatic, in the sense
of \cite{Mousavi&Schaeffer&Shallit:2016}, so we can apply the {\tt Walnut} theorem-prover \cite{Mousavi:2016} to establish this claim.  For more about {\tt Walnut}, see \cite{Shallit:2022}.  

We start with the Fibonacci automaton for
$\varphi^\omega(0)$, as displayed in Figure~\ref{ffaut}.
\begin{figure}[H]
\begin{center}
    \includegraphics[width=5in]{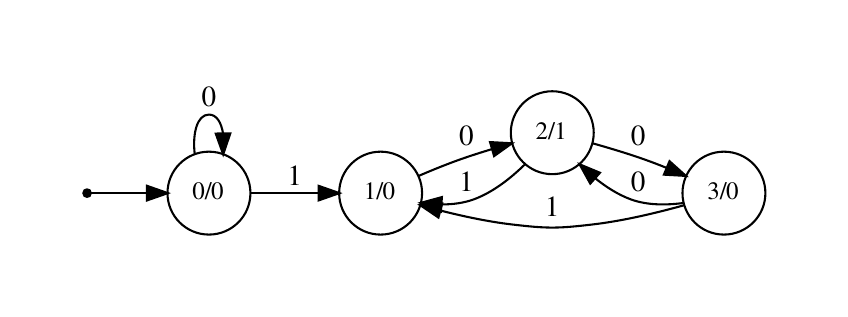}
\end{center}
\caption{Fibonacci automaton for $\varphi^\omega(0)$.}
\label{ffaut}
\end{figure}

Let us name the above automaton {\tt FF.txt}, and store it in the {\tt Word Automata Library} of
{\tt Walnut}.
We can verify the correctness of this automaton as follows.  First, we claim that $\varphi^\omega(0)=0{\bf f}$.  To see this, let $f$ denote the morphism that maps $0\to 010, 1\to 01$; i.e., the morphism $f$
is the square of the Fibonacci morphism.  One can easily verify the identities
$\varphi^n(0) = 0f^n(0)0^{-1}$ and $\varphi^n(01) = 0f^n(10)0^{-1}$ by simultaneous induction,
whence follows the claim.  We can then use Walnut to verify the correctness of the automaton
{\tt FF} with the command
\begin{verbatim}
eval verifyFF "?msd_fib FF[0]=@0 & Ai FF[i+1]=F[i]":
\end{verbatim}
which returns {\tt TRUE}.

Now we can use {\tt Walnut} to create a Fibonacci automaton for $\bf w$.
\begin{verbatim}
morphism g "0->01 1->11":
image GF g FF:
\end{verbatim}
The resulting automaton is called {\tt GF.txt} and is displayed in Figure~\ref{gfaut}.
\begin{figure}[H]
\begin{center}
    \includegraphics[width=5.5in]{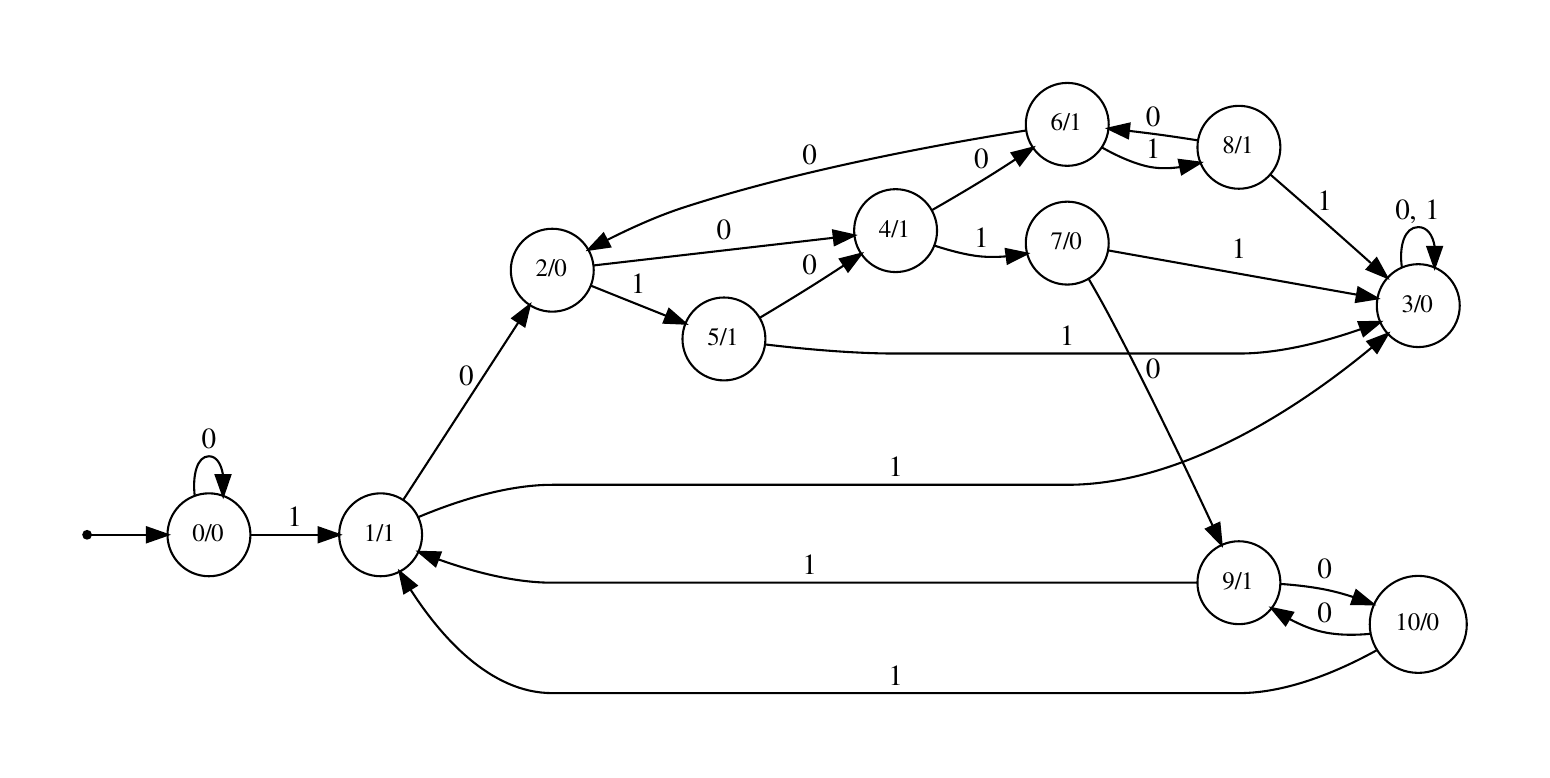}
\end{center}
\caption{Fibonacci automaton for ${\bf w} = g(\varphi^\omega(0))$.}
\label{gfaut}
\end{figure}

\begin{theorem}
The word ${\bf w}$ does not contain antisquares other than $01$ and $10$, and has critical exponent $2 + \alpha$.
\end{theorem}

\begin{proof}
We write a Walnut formula asserting that there exists an
antisquare of order $\geq 2$, as follows:
\begin{verbatim}
eval antisq "?msd_fib Ei,n (n>=2) & At (t<n) => GF[i+t]!=GF[i+n+t]":
\end{verbatim}
This returns {\tt FALSE}, so there are no antisquares of order $\geq 2$.

We now compute the periods that are associated with factors that have exponent $\geq 3$.
\begin{verbatim}
eval gfper "?msd_fib Ei (p>=1) & (Aj (j<=2*p) => GF[i+j]=GF[i+j+p])":
\end{verbatim}
The predicate above produces the automaton in Figure~\ref{gfper}, which shows that these periods are of the form
$10010^*$ in Fibonacci representation.

\begin{figure}[ht]
    \center{\includegraphics[width=0.7\textwidth]{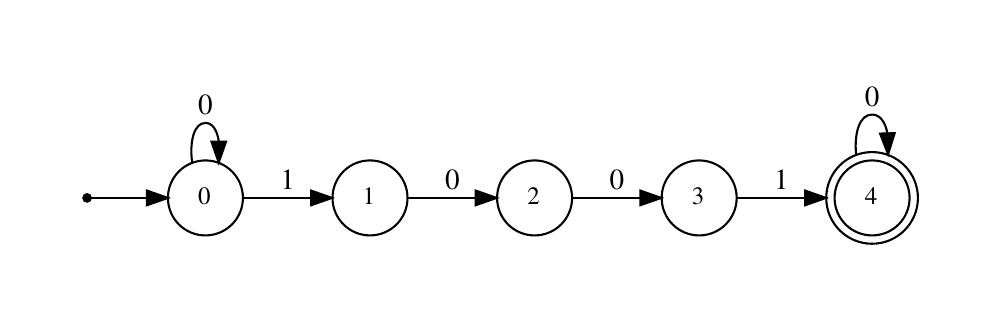}}
    \caption{Automaton for periods associated with $3^+$-powers in ${\bf w}$.}
    \label{gfper}
\end{figure}

Next, we compute the pairs $(n,p)$ such that ${\bf w}$ contains a factor of length $n+p$ with period $p$ of the form 
$10010^*$ and $n+p$ is the longest length of any factor with this period.
\begin{verbatim}
reg pows msd_fib "0*10010*";
def maximalreps "?msd_fib Ei
    (Aj (j<n) => GF[i+j] = GF[i+j+p]) & (GF[i+n] != GF[i+n+p])":
eval highestpow "?msd_fib (p>=1) & $pows(p) &
    $maximalreps(n,p) & (Am $maximalreps(m,p) => m <= n)":
\end{verbatim}

The automaton produced by the predicate {\tt highestpow} is given in Figure~\ref{highestpow}.
\begin{figure}[ht]
    \center{\includegraphics[width=\textwidth]{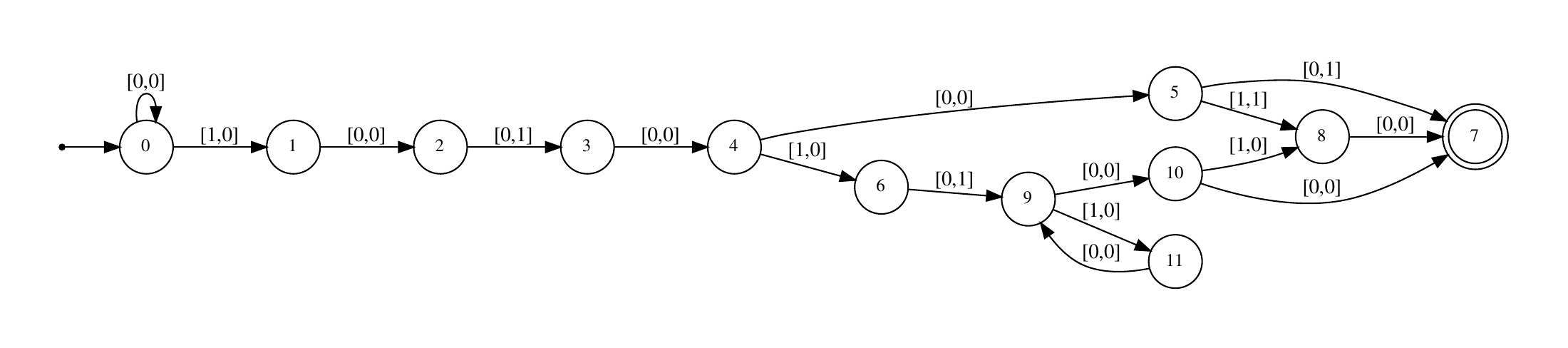}}
    \caption{Automaton for pairs $(n,p)$ associated with highest powers in ${\bf w}$.}
    \label{highestpow}
\end{figure}

The strings accepted by this automaton, omitting the leading
$[0,0]$, are as follows:
\begin{itemize}
\item $[1,0][0,0][0,1][0,0][0,0][0,1]$
\item $[1,0][0,0][0,1][0,0][0,0][1,1][0,0]$
\item $[1,0][0,0][0,1][0,0][1,0][0,1]([1,0][0,0])^k[0,0][1,0][0,0]$, $k \geq 0$
\item $[1,0][0,0][0,1][0,0][1,0][0,1]([1,0][0,0])^k[0,0][0,0]$, $k \geq 0$.
\end{itemize}
These correspond, respectively, to the values
\begin{itemize}
\item $(n,p) = (13,6) = (2F_6-3, 2F_4)$
\item $(n,p) = (23,10) = (2F_7 -3, 2F_5)$
\item $(n,p) = (F_{2k+10} + F_3 + \sum_{3\leq i\leq k+3} F_{2i}, F_{2k+8} + F_{2k+5}) = (2 F_{2k+9} -3, 2F_{2k+7})$ for $k \geq 0$
\item $(n,p) = (F_{2k+9} + \sum_{3\leq i\leq k+3} F_{2i-1}, F_{2k+7} + F_{2k+4})= (2F_{2k+8} - 3, 2F_{2k+6})$ for $k \geq 0$.
\end{itemize}
where we have used the well-known Fibonacci identities
\begin{align*}
F_2 + F_4 + F_6 + \cdots + F_{2t} &=  F_{2t+1} - 1 \\
F_1 + F_3 + F_5 + \cdots + F_{2t-1} &= F_{2t}.
\end{align*}
Now the exponent of these finite factors is $(n+p)/p$, which is
$$\frac{2 F_j + 2F_{j-2} - 3}{2 F_{j-2}} $$
for $j \geq 6$.   These quotients tend to $2 + \alpha$ from below,
and hence the critical exponent is $2+ \alpha$.
\end{proof}

\section{Optimality of the previous construction}
\label{three}

In this section we show that the critical exponent of the word constructed in Section~\ref{two} is
best possible; i.e., that every infinite good word has critical exponent
at least $2+\alpha$.  It is somewhat easier to work with bi-infinite words,
so we begin with results concerning bi-infinite words and then explain at the end of the
section how to obtain the desired result for right-infinite words.

If $S$ is a set of nonempty finite words, then by ${{}^\omega}S^\omega$ 
we mean the set of bi-infinite
words made up of concatenations of
the elements of $S$.

\begin{theorem}\label{thmfib}
Every bi-infinite binary word avoiding $4$-powers and $\{11,000,10101\}$
has the same set of factors as $\bf f$.
\end{theorem}

\begin{proof}
First, we check that $\bf f$ avoids $4$-powers and $\{11,000,10101\}$.

Now consider a bi-infinite binary word $\bf w$ avoiding $4$-powers and $\{11,000,10101\}$.
Since $\bf w$ avoids $11$, we have ${\bf w}\in{}^\omega\{01,0\}^\omega$.
Thus there exists a bi-infinite word $\bf v$ such that ${\bf w}=h(\bf v)$, where $h$ is the morphism $0 \rightarrow 01$, $1 \rightarrow 0$.
Now it suffices to show that the pre-image $\bf v$ also avoids $4$-powers and $\{11,000,10101\}$.
Clearly, $\bf v$ avoids $4$-powers, since otherwise ${\bf w}=h(\bf v)$ would contain a 4-power.
Now we show by contradiction that $\bf v$ avoids every factor in $\{11,000,10101\}$.
\begin{itemize}
     \item If $\bf v$ contains $11$, then $\bf v$ contains $110$.
     So $\bf w$ contains $h(110)=0001$---a contradiction, since $\bf w$ avoids $000$.
     \item If $\bf v$ contains $000$ then $\bf w$ contains $h(000)=010101$---a contradiction, since $\bf w$ avoids $10101$.
     \item If $\bf v$ contains $10101$ then $\bf v$ contains $0101010$, since $\bf v$ avoids $11$.
     So $\bf w$ contains $h(0101010)=01001001001$.
     Since $\bf w$ avoids $11$, we see that $\bf w$ contains $010010010010=(010)^4$---a contradiction, since $\bf w$ avoids 4-powers.
\end{itemize}
\end{proof}

\begin{lemma}\label{lemgfib}
 Every bi-infinite binary word avoiding $4$-powers and
 \begin{align*}
     F &=\{0011,0110,1100,1001,010101,101010,1000101110, \\
     & 0111010001,101110111011101,010001000100010\}
 \end{align*}
 has the same set of factors as $g(\bf f)$ or $\overline{g(\bf f)}$.
\end{lemma}

\begin{proof}
Notice that $F$ is closed under bitwise complement and reversal.
Let $\bf w$ be a bi-infinite binary word avoiding $4$-powers and $F$.
Suppose that $\bf w$ contains $001011$.
\begin{itemize}
    \item Since $1001\in F$ and $0110\in F$, the word $\bf w$ contains $00010111$.
    \item Since $0000$ and $1111$ are $4$-powers, the word $\bf w$ contains $1000101110$.
\end{itemize}
This is a contradiction since $1000101110\in F$.
So $\bf w$ avoids $001011$. By considering the complement, the word $\bf w$ also avoids $110100$.
Now suppose that $\bf w$ contains $110111011$.
\begin{itemize}
    \item Since $0110\in F$, the word $\bf w$ contains $11101110111$.
    \item Since $1111$ is a $4$-power, the word $\bf w$ contains $0111011101110$.
    \item Since $0011\in F$ and $1100\in F$, the word $\bf w$ contains $101110111011101$.
\end{itemize}
This is a contradiction since $101110111011101\in F$.
So $\bf w$ avoids $110111011$. By considering the complement, the word $\bf w$ also avoids $001000100$.
Thus, $\bf w$ avoids $4$-powers and
$$F'=\{0011,0110,1100,1001,010101,101010,001011,110100,110111011,001000100\}.$$
Notice that $F'$ is closed under complement and reversal.
By symmetry, we now suppose that $\bf w$ contains $11$.
Notice that $111$, $11011$, and $1101011$ are the only possible factors of $\bf w$
that start with $11$, end with two identical letters,
and contain two identical letters only as a prefix and a suffix.
In particular, the word $\bf w$ avoids $00$.
Moreover, the blocks of consecutive $1$'s have length 1 or 3.
So ${\bf w}\in{}^\omega\{01,11\}^\omega$.
Thus ${\bf w}=g(\bf v)$ for some bi-infinite binary word $\bf v$.

Since $\bf w$ avoids $4$-powers, the word $\bf v$ also avoids 4-powers.
Moreover,
\begin{itemize}
    \item $g(11)=1111$ is a $4$-power,
    \item $g(000)=010101$ belongs to $F$,
    \item $g(10101)=1101110111$ contains $110111011\in F'$.
\end{itemize}
So $\bf v$ also avoids $\{11,000,10101\}$.
By Theorem~\ref{thmfib}, the word $\bf v$ has the same set of factors as $\bf f$.
That is, the word $\bf w$ has the same set of factors as $g(\bf f)$.
\end{proof}

\begin{theorem}\label{thm_lower_bd_ce}
Every good bi-infinite binary word has critical exponent at least $2+\alpha$.
\end{theorem}

\begin{proof}
Suppose that $\bf w$ is a good bi-infinite binary word; that is,
it contain no antisquares except possibly $01$ and $10$.   Also
assume $\bf w$ has critical exponent smaller than $2+\alpha$.
Consider the set $F$ from Lemma~\ref{lemgfib} 
and notice that $F\setminus\{101110111011101,010001000100010\}$ contains only antisquares.
So $\bf w$ avoids $4$-powers and $F\setminus\{101110111011101,010001000100010\}$.
Moreover, $\bf w$ avoids $101110111011101=(1011)^{15/4}$ since $15/4 > 2+\alpha$.
By symmetry, $\bf w$ also avoids $010001000100010$.
So $\bf w$ avoids $4$-powers and $F$.

By Lemma~\ref{lemgfib}, $\bf w$ has the same set of factors as either $g(\bf f)$ or $\overline{g(\bf f)}$.
So ${\bf w}$ has critical exponent $2+\alpha$.
\end{proof}

\begin{corollary}
Every (right-) infinite good binary word 
has critical exponent at least $2+\alpha$.
\end{corollary}

\begin{proof}
Let ${\bf w}$ be an infinite good binary word,
and let $\mbox{RecFac}({\bf w})$
denote the set of its recurrent factors.  That is, the set $\mbox{RecFac}({\bf w})$ consists of the factors
of ${\bf w}$ that appear infinitely often in ${\bf w}$.  Then for any $y \in \mbox{RecFac}({\bf w})$, we see that
$y$ has arbitrarily large two-sided extensions in $\mbox{RecFac}({\bf w})$.  By a `two-sided' analogue of K\"onig's infinity lemma,
there exists a bi-infinite word ${\bf w}'$ such that every factor of ${\bf w}'$ is an element of $\mbox{RecFac}({\bf w})$.
By Theorem~\ref{thm_lower_bd_ce}, the bi-infinite word ${\bf w}'$ has critical exponent at least $2+\alpha$,
and thus so does the infinite word ${\bf w}$.
\end{proof}

\section{The class \texorpdfstring{$\AO_\ell$}{AO}}
\label{AO_l}

Instead of avoiding all antisquares of order greater than one, we could consider avoiding arbitrarily large antisquares.

\begin{proposition}\label{7/3}
Every infinite binary word avoiding
$\tfrac{7}{3}$-powers contains arbitrarily large antisquares.
\end{proposition}

\begin{proof}
By a result of \cite{Karhumaki&Shallit:2004}, we know
that every infinite binary word avoiding
$\tfrac{7}{3}$-powers can be written in the
form $x_1 \mu(x_2 \mu(x_3 \mu(\cdots)) \cdots)$, where $x_i \in \{\epsilon, 0, 1, 00, 11 \}$ and $\mu$ is the Thue-Morse morphism, defined by $\mu:0 \rightarrow 01$, $1 \rightarrow 10$.  It follows that every such word must contain arbitrarily large factors of the form $\mu^n(0)$.  But every word $\mu^n(0)$ for $n \geq 1$
is an antisquare.
\end{proof}

On the other hand, we can prove the following result on the class $\AO_\ell$ of binary words
avoiding antisquares of order $\geq \ell$:

\begin{theorem}\label{thm_AO}
There exists an infinite $\beta^+$-free binary word containing no antisquare of order $\geq \ell$ for the following pairs $(\ell,\beta)$:
\begin{itemize}
    \item[(a)] $(2,2+\alpha)$
    \item[(b)] $(3,8/3)$
    \item[(c)] $(5,5/2)$
    \item[(d)] $(6,7/3)$
\end{itemize}
These are all optimal.
\end{theorem}

\begin{proof}
Item (a) was already proved in Section~\ref{two}.  For each of the remaining pairs $(\ell,\beta)$, we apply a morphism $\xi_\ell$ to any ternary squarefree infinite word ${\bf w}$ and check that it has the desired properties.
The morphisms are given in Table~\ref{tabOrderMorphisms}.  The
columns are $\ell$, where the word contains no antisquares of order $\geq \ell$; $\beta$, where the word avoids $\beta^+$ powers;
$s$, the size of the uniform morphism; and the morphism.

\begin{table}[H]
\begin{center}
\begin{tabular}{|c|c|c|c|l|}
\hline
$\ell$ & $\beta$ & $s$ & morphism name & morphism\\
\hline
3 & 8/3 & 36 & $\xi_3$ &  $0 \rightarrow 001001010011001010010011001001010011$  \\
&&&&    $1 \rightarrow 001001010010011001010011001010010011$  \\
&&&&    $2 \rightarrow 001001010010011001001010011001010011$\\
\hline
5 & 5/2 & 19 & $\xi_5$ & $0 \rightarrow 0010110100101101011$\\
&&&& $1  \rightarrow 0010110100101100101$\\
&&&& $2  \rightarrow 0010110011001010011$ \\
\hline
6 & 7/3 & 37 & $\xi_6$ & $0 \rightarrow 0010011010010110100110110011010011011$\\
&&&& $1 \rightarrow 0010011010010110100110110010011010011$\\
&&&& $2 \rightarrow 0010011010010110100110110010011001011$\\
\hline
\end{tabular}
\end{center}
\caption{Morphisms generating words in $\AO_\ell$.}
\label{tabOrderMorphisms}
\end{table} 

To verify the $\beta^+$-freeness of $\xi_\ell({\bf w})$
we use~\cite[Lemma~23]{Mol&Rampersad&Shallit:2020}.  In order to do so,
we check that $\xi_\ell$ is synchronizing and that $\xi_\ell(u)$ is $\beta^+$-free for every squarefree
ternary word $u$ of length $t$, where $t$ is specified by~\cite[Lemma~23]{Mol&Rampersad&Shallit:2020}.
In order to show that $\xi_\ell(\bf{w})$ contains no antisquares of order $\geq \ell$, we find a length
$m$ such that if $v$ and its complement are both factors of $\xi_\ell(\bf{w})$, then $|v|\leq m$.
We can then check that $\xi_\ell(\bf{w})$ contains no antisquares of order $\geq \ell$ by exhaustively checking all factors of length at most $2m$.  The parameters $t$ and $m$
for each $\xi_\ell$ are given in Table~\ref{tabOrderChecks}.
\begin{table}[H]
\begin{center}
\begin{tabular}{|c|c|c|c|c|}
\hline
morphism & $\ell$ & $\beta$ & $t$ & $m$ \\
\hline
$\xi_3$ & 3 & 8/3 & 8 & 6\\
$\xi_5$ & 5 & 5/2 & 10 & 16\\
$\xi_6$ & 6 & 7/3 & 14 & 26\\
\hline
\end{tabular}
\end{center}
\caption{Parameters for checking correctness of $\xi_\ell$.}
\label{tabOrderChecks}
\end{table}

The optimality of item (a) was already proved in Section~\ref{three}.  The optimality of the remaining items can be established by depth-first search.
For each pair $(\ell,\beta)$, a longest word containing no antisquares of order $\geq \ell$, but avoiding $\beta$-powers (instead of $\beta^+$), is given in
Table~\ref{tabOrderOptimality}.  The columns give $\ell$, $\beta$, the length $L$ of a
longest such word, and a longest such word.
\begin{table}[H]
\begin{center}
\begin{tabular}{|c|c|l|l|}
\hline
$\ell$ & $\beta$ & $L$ & example \\
\hline
4  & $8/3$ &  29 & 00100101001100101001100110100\\
\hline
5 & $5/2$ & 32 & 00100101100101101001011001011011\\
\hline
6 & $7/3$ & 30 & 001011001101001011010011001011\\
\hline
\end{tabular}
\end{center}
\caption{Longest words avoiding $\beta$-powers and containing no antisquares of order $\geq \ell$.}
\label{tabOrderOptimality}
\end{table}

\end{proof}

\section{The class \texorpdfstring{$\AN_n$}{AN}}
\label{AN_n}

In this section, we consider the class $\AN_n$ of binary words containing no more than $n$ distinct antisquares as factors.

\begin{theorem}\label{thm_AN}
There exists an infinite $\beta^+$-free binary word containing
no more than $n$ antisquares for the following pairs $(n, \beta)$:
\begin{itemize}
    \item[(a)] $(2,2+\alpha)$
    \item[(b)] $(3,3)$
    \item[(c)] $(6,8/3)$
    \item[(d)] $(9,38/15)$
    \item[(e)] $(10,5/2)$
    \item[(f)] $(15,17/7)$
    \item[(g)] $(16,7/3)$.
\end{itemize}
These are all optimal.
\end{theorem}

\begin{proof}
Item (a) was already proved in Section~\ref{two}.
For the remaining cases, the proof is similar to that of Theorem~\ref{thm_AO}:
for each pair $(n,\beta)$, we apply a morphism $\zeta_n$ to any ternary squarefree infinite word
${\bf w}$ and check that it has the desired properties.
The morphisms are given in Table~\ref{tab5}.  The
columns are $n$, the largest number of allowed
antisquares; $\beta$, where the word avoids $\beta^+$ powers;
$s$, the size of the uniform morphism; and the morphism.
\begin{table}[H]
\begin{center}
\resizebox{.87\columnwidth}{!}{%
\begin{tabular}{|c|c|c|l|l|}
\hline
$n$ & $\beta$ & $s$ & morphism & morphism\\
&&& name &  \\
\hline
3 & 3 & 13 & $\zeta_3$ & $0 \rightarrow 0010001000101$ \\
&&&& $1 \rightarrow 0001001000101$ \\
&&&& $2 \rightarrow 0001000100101$ \\
\hline
6 & $8/3$ & 36 & $\zeta_6$ & $0 \rightarrow 001001010011001010010011001001010011$ \\
&&&& $1 \rightarrow 001001010010011001010011001010010011$ \\
&&&&$2 \rightarrow 001001010010011001001010011001010011$ \\
\hline
9 & $38/15$ & 192 & $\zeta_9$ & 
$0 \rightarrow 
0010100101100110010100101100110010110010100101100110101100110010$\\
&&&&\quad\ \ \ 1100101001011001100101100101001011001101011001100101100101001011\\
&&&&\quad\ \ \ 0011001010011001010010110011001011001010010110011010110011001011 \\
&&&& $1 \rightarrow 0010100101100110010100101100110010110010100101100110101100110010$\\
&&&&\quad\ \ \ 1100101001011001100101001100101001011001100101100101001011001101\\
&&&&\quad\ \ \ 0110011001011001010010110011001011001010010110011010110011001011 \\
&&&& $2 \rightarrow 0010100101100110010100101100110010110010100101100110010100110010$\\
&&&&\quad\ \ \ 1001011001100101100101001011001100101001011001100101100101001011\\
&&&&\quad\ \ \ 0011001010011001010010110011001011001010010110011010110011001011 \\
\hline
10 & $5/2$ & 75 & $\zeta_{10}$ & $0 \rightarrow 0010100101100110010100110010100101100110010110010100101100110101$\\
&&&&\quad\ \ \ 10011001011 \\
&&&& $1 \rightarrow
0010100101100110010100110010100101100110010110010100101100110010$\\
&&&&\quad\ \ \ 10011001011 \\
&&&& $2 \rightarrow 
0010100101100110010100110010100101100110010100110010110010100101$\\
&&&&\quad\ \ \ 10011001011 \\
\hline 
15 & $17/7$ & 194 & $\zeta_{15}$ & $0 \rightarrow
00100110010110010011010010110010011001011001001101001011001001101$\\
&&&&\quad\ \ \ 00110010011010010110010011010011001001100101100100110100101100100\\
&&&&\quad\ \ \ 1101001100100110100101100100110010110010011010010110010011010011 \\
&&&& $1 \rightarrow 
00100110010110010011010010110010011001011001001101001011001001101$\\
&&&&\quad\ \ \ 00110010011010010110010011001011001001101001011001001101001100100\\
&&&&\quad\ \ \ 1100101100100110100101100100110100110010011010010110010011010011 \\
&&&& $2 \rightarrow 00100110010110010011010010110010011001011001001101001011001001101$\\
&&&&\quad\ \ \ 00110010011001011001001101001011001001101001100100110100101100100\\
&&&&\quad\ \ \ 1100101100100110100101100100110100110010011010010110010011010011 \\
\hline
16 & $7/3$ & 192 & $\zeta_{16}$ & $0 \rightarrow 0010011001011001001101001011001001101001100100110100101100110100$\\
&&&&\quad\ \ \ 1011001001101001100100110100101100100110010110010011010010110010\\
&&&&\quad\ \ \ 0110100110010011010010110010011001011001001101001011001101001011 \\
&&&& $1 \rightarrow
0010011001011001001101001011001001101001100100110100101100100110$\\
&&&&\quad\ \ \ 0101100100110100101100110100101100100110100110010011010010110010\\
&&&&\quad\ \ \ 0110010110010011010010110010011010011001001101001011001101001011 \\
&&&& $2 \rightarrow 0010011001011001001101001011001001101001100100110100101100100110$\\
&&&&\quad\ \ \ 0101100100110100101100100110100110010011010010110011010010110010\\
&&&&\quad\ \ \ 0110010110010011010010110011010010110010011010011001001101001011 \\
\hline
\end{tabular}
}
\end{center}
\caption{Morphisms generating words in $\AN_n$.}
\label{tab5}
\end{table}   

To verify the $\beta^+$-freeness of $\zeta_n({\bf w})$
we use~\cite[Lemma~23]{Mol&Rampersad&Shallit:2020}.  In order to do so,
we check that $\zeta_n$ is synchronizing and that $\zeta_n(u)$ is $\beta^+$-free for every squarefree
ternary word $u$ of length $t$, where $t$ is specified by~\cite[Lemma~23]{Mol&Rampersad&Shallit:2020}.
In order to show that $\zeta_n(\bf{w})$ contains at most $n$ distinct antisquares, we find a length
$m$ such that if $v$ and its complement are both factors of $\zeta_n(\bf{w})$, then $|v|\leq m$.
We can then enumerate the antisquares appearing in
$\zeta_n(\bf{w})$ and check that there are at most $n$ of them.  The parameters $t$ and $\ell$
for each $\zeta_n$ are given in Table~\ref{tab5b}.
\begin{table}[H]
\begin{center}
\begin{tabular}{|c|c|c|c|c|}
\hline
morphism & $n$ & $\beta$ & $t$ & $m$ \\
\hline
$\zeta_3$ & 3 & 3 & 6 & 4\\
$\zeta_6$ & 6 & 8/3 & 8 & 6\\
$\zeta_9$ & 9 & 38/15 & 9 & 17\\
$\zeta_{10}$ & 10 & 5/2 & 10 & 17\\
$\zeta_{15}$ & 15 & 17/7 & 11 & 12\\
$\zeta_{16}$ & 16 & 7/3 & 14 & 13\\
\hline
\end{tabular}
\end{center}
\caption{Parameters for checking correctness of $\zeta_n$.}
\label{tab5b}
\end{table}   

The optimality of item (a) was already proved in Section~\ref{three}.  The optimality of the remaining items can be established by depth-first search.
For each pair $(n,\beta)$, a longest word containing at most $n$ antisquares, but avoiding $\beta$-powers (instead of $\beta^+$), is given in
Table~\ref{tab6}.  The columns give $n$, $\beta$, the length $L$ of a
longest such word, and a longest such word.
\begin{table}[H]
\begin{center}
\begin{tabular}{|c|c|l|l|}
\hline
$n$ & $\beta$ & $L$ & example \\
\hline
5  & 3 &  17 & 00101001010010011\\
\hline
8 & $8/3$ & 52 & 0010010100110010100110011010011001101011001101011011\\
\hline
9 & $38/15$ & 407 & 00100101001101001010011010011001101001010011001010011\\
&&&00110100101001101001100110101100110100101001101001100\\
&&&11010010100110010100110011010010100110100110011010110\\
&&&01101001010011010011001101001010011010011001101011001\\
&&&10100101001101001100110100101001100101001100110100101\\
&&&00110100110011010110011010010100110100110011010010100\\
&&&11001010011001101001010011010011001101001010011001010\\
&&&011001101001010011001101001101011011\\
\hline
14 & $5/2$ & 92 & 001101001011001101100110100101100110110011010011 \\
&&& 01100110100101100110110011010011011001101100\\
\hline
15 & $17/7$ & 156  &0010110011010010110010011010011001001101001011001001\\
&&&1001011001001101001011001001101001100100110100101100\\
&&&1001100101100100110100101100100110010110010011001001\\
\hline
16 & 7/3 & 38 & 00101100101101001011001101001011001011 \\
\hline
\end{tabular}
\end{center}
\caption{Longest words with at most $n$ antisquares and avoiding $\beta$-powers.}
\label{tab6}
\end{table}
\end{proof}

\section{Minimal antisquares}
\label{manti}

Consider the language $L$ of all finite good words.
In this section we determine the
\emph{minimal antisquares} or \emph{minimal forbidden factors} 
for $L$ \cite{Mignosi&Restivo&Sciortino:2002}.
These are the words $w$ such that $w$ is an antisquare,
but $w$ properly contains no antisquare factors, except possibly $01$ and $10$.
This characterization will be useful for enumerating the number of length-$n$
words in $L$.

The goal is to prove the following theorem:
\begin{theorem}
The minimal antisquares are, organized by order $n$, as follows:
\begin{itemize}
\item $n = 1$:  $\{ 01, 10 \}$
\item $n = 2$:  $\{ 0011,0110, 1001, 1100 \}$
\item $n = 3$: $\{ 010101, 101010 \}$
\item $n = 4$:  $\emptyset$
\item $n \geq 5$:  all the $2n$ conjugates (cyclic shifts) of
$0^{n-2} 1 0 1^{n-2} 0 1$.
\end{itemize}
\label{minimal}
\end{theorem}

We start with some basic results about antisquares.

\begin{lemma}
If $x$ is an antisquare, so is every conjugate of $x$.
\label{minanti1}
\end{lemma}

\begin{proof}
Write $x$ as $ay \overline{a} \overline{y}$ for some (possibly empty) word $y$.
Then a cyclic shift by one symbol gives $y \overline{a} \overline{y} a$,
which is clearly an antisquare.  Repeating this argument $|x|$ times gives the result.
\end{proof}

\begin{lemma}
A word $w$ is a minimal antisquare if and only if all conjugates of $w$ are minimal antisquares.
\label{minanti2}
\end{lemma}

\begin{proof}
Let $w=u\bu$. Let us prove the forward direction first.
Assume, contrary to what we want to prove, that $w$ is a minimal antisquare, but some rotation of $w$ contains a shorter antisquare, 
other than $01$ and $10$. Let $rs$ be an antisquare of minimum length with $|rs|>2$ that is not a factor of $w$, but is a factor of some rotation of $w$. Then we can write $w=sxr$ with $r,s,x\neq \epsilon$. There are three cases to consider:

\medskip

$|r|\ge|\bu|$: Let $t$ be defined
by $r=t\bu$.
    Then $w=u\bu=sxt\bs\bx\bt$, where $rs=t\bs\bx\bt s$ is an antisquare. If $t=\epsilon$, then $rs=\bs\bx s$; hence $w=sx\bs\bx$ contains the antisquare $sx\bs$. This is a contradiction, since we assumed $w$ does not have a shorter antisquare, but $sx\bs$ is an antisquare in $w$ with $s,x\neq\epsilon$. Therefore $t\neq\epsilon$.
    
    Since $w$ and $rs$ are both antisquares, they have even length. Therefore $|x|$ is even. Consider the factorization $x=x_1x_2$ with $|x_1|=|x_2|$. Then $rs=t\bs\bx\bt s = t\bs\bx_1\bx_2\bt s$ is an antisquare, where $|t\bs\bx_1| = |\bx_2\bt s|$. Therefore, $t\bs\bx_1$ must end with $t\bs$ and $\bx_2\bt s$ must begin with $\bt s$, giving the antisquare $t\bs\bt s$, which is not $01$ or $10$ (since $t,s\neq\epsilon$), and is shorter than $rs$. This contradicts our assumption that $rs$ is the smallest such antisquare.
    
    \medskip
    
    $|s|\ge|u|$: Let $s$ be defined by $s=ut$.
    Then $w=u\bu=\bt\bx\br txr$.
    Now $rs=r\bt\bx\br t$ is an antisquare, and hence the complement $\br\bs=\br txr\bt$ is also an antisquare.
    Let $r'=\br txr$ and $s'=\bt$. So we can write $w=\bt\bx\br txr = s'\bx r'$ where $r's'$ is an antisquare of the same length as $rs$, and $|r'| > |\bu|$. This reduces the problem to the previous case.
    
    \medskip
    
    $|s|<|u|,|r|<|\bu|$: In this case we can write $u=sy$ and $\bu=zr$ where $y,z\neq\epsilon$.  This means that $u$ ends in $\br$ and $\bu$ begins with $\bs$, implying that $w$ contains the antisquare $\br\bs$. This contradicts the assumption that $w$ does not contain a shorter antisquare.

\medskip

The reverse direction is easy. If $w$ contains a shorter antisquare, other than $01$ and $10$, then at least one rotation of $w$ also has the same antisquare.
\end{proof}

We
introduce some terminology.  A {\it run\/} in a word is a maximal block of consecutive identical symbols.   The {\it run-length encoding\/}
$r: \Sigma^* \rightarrow \naturals^*$, where
$\naturals = \{ 1,2,3,\ldots \}$
is a map sending a word $x$ to the list of the
lengths of the consecutive runs in $x$.
For example, $r({\tt access}) = 1212$.

\begin{lemma}
If a nonempty word $x$ is an antisquare, then the number of runs it contains must be
congruent to $2$ or $3$ $({\rm mod}\ 4)$.
\label{runlength}
\end{lemma}

\begin{proof}
Write $x = u \overline{u}$ and consider 
the runs in $u$.  If $u$ has $2k+1$ runs,
then so does $\overline{u}$.  Furthermore,
$u$ ends in a different letter
than the start of $\overline{u}$.  Hence
$x$ has $4k+2$ runs.

Otherwise $u$ has $2k$ runs.  Then the last letter of $u$ is the same
as the first letter of $\overline{u}$.
Hence
$x$ has $4k-1$ runs.
\end{proof}

\begin{lemma}
The word $0^k 10 \, 1^k 01$ is a minimal antisquare for $k \geq 3$.
\label{minanti3}
\end{lemma}

\begin{proof}
It is easy to see that
$x = 0^k 10 \, 1^k 01$ is an antisquare.  Suppose $k \geq 3$ and  suppose $x$ has an antisquare proper factor $w$ other than $01$ and $10$.  Now $x$ has six runs, so by Lemma~\ref{runlength} we know that $w$ has either two, three, or six runs.

If $w$ has two runs, it must
be of the form $a^i \overline{a} ^i$ for
$i \geq 2$, but inspection shows that $x$ has no factor of that form.  

If $w$ has three runs, it must be of the form $a^i \overline{a}^{i+j} a^j$ for some $i, j \geq 1$.   The only possibility
is $i = 1$, $i+j=k$, $j=1$,
which forces $k = 2$, a contradiction.  

Finally, if $w$ has six runs,
then $w = 0^{\ell} 1 0 1^{\ell} 0 1$ for some $\ell < k$, which would not be a factor of $x$.
\end{proof}

We are now ready to complete the proof
of Theorem~\ref{minimal}.   We say $x$ has an {\it interior occurrence\/} of $y$ if we can write
$x = wyz$ for nonempty words $w,z$.

\begin{proof}[Proof of Theorem~\ref{minimal}.]
By combining Lemmas~\ref{minanti1}, \ref{minanti2}, and \ref{minanti3}, and verifying the listed cases for
$n \leq 4$, we see that all the words given in the statement of the theorem are minimal antisquares.

It now remains to see that there are no other minimal antisquares.  The idea is to classify antisquares $x$ by the number of runs.  In what follows, we assume, without loss of generality, that $x$ begins with $0$.

\medskip

Two runs:   then $x = 0^i 1^i$ for some
$i \geq 1$.  If $i \geq 3$ then $x$ contains
the antisquare $0011$.  So the only minimal antisquares are $01$ and $0011$.

\medskip

Three runs:  then $x = 0^i 1^{i+j} 0^j$.  If either $i$ or $j$ is at least $2$,
then $x$ contains the antisquare $0011$ or $1100$.
So the only minimal antisquare is $0110$.

\medskip

Six runs:  then $x = 0^i 1^j 0^k 1^i 0^j 1^k$.
If $i,j \geq 2$ then $x$ contains the antisquare
$0011$, and similarly if $j,k \geq 2$ and
$k,i \geq 2$.  
It follows that $(i,j,k) \in \{ (1,1,1), (1,1,n), (1,n,1), (n,1,1) \}$ for
$n \geq 2$.  The cases $(1,1,2), (1,2,1), (2,1,1)$ are ruled
out by an antisquare of the form $1001$ or
$0110$.  So $(i,j,k) \in \{ (1,1,1), (1,1,n), (1,n,1), (n,1,1) \}$ for
$n \geq 3$.   The case $(1,1,1)$ corresponds
to the word $010101$, and the remaining
cases correspond to certain conjugates of
$0^n 1 0 1^n 0 1$ for $n \geq 3$, already listed
in the statement of the theorem.

\medskip

Seven runs:  then $x = 0^i 1^j 0^k 1^{i+l} 0^j 1^k 0^l$.  Again, if $i,j \geq 2$, or $j,k \geq 2$, or $k, l \geq 2$, then $x$ contains a shorter antisquare $0011$ or $1100$.  Since $i+l \geq 2$, the same argument rules out $k \geq 2$ and $j \geq 2$.  So the only cases remaining
are $$(i,j,k,l) \in \{(1,1,1,1), (1,1,1,n),(n,1,1,1), (n,1,1,n) \}$$ for
$n \geq 2$.  The first case $(1,1,1,1)$
corresponds to $010110101$, which has the
antisquare $0110$, and it is easy to verify that the remaining cases are certain conjugates of $0^{n+1} 1 0 1^{n+1} 0 1$ for $n \geq 3$, already listed in the statement of the theorem.

It now remains to handle the case of more than $7$ runs.  By Lemma~\ref{runlength} $x$
has at least 10 runs.
This involves a rather tedious examination of cases, based on the following three simple observations: 
\begin{itemize}
    \item[(a)]
if $r(x)$ contains two consecutive terms, both $\geq 2$, then $x$
contains the shorter antisquare $0011$ or $1100$;
    \item[(b)] if $r(x)$ contains six
    consecutive terms $a1bc1d$ with
    $a \geq c$ and $b \leq d$, then 
    $x$ contains the shorter antisquare
    $0^c 1 0^b 1^c 0 1^b$ or its complement.
    
    \item[(c)] if $r(x)$ contains an interior 
    occurrence of $2$, then $x$ contains
    the antisquare $0110$ or $1001$.
\end{itemize}

Suppose $x = u \overline{u}$.  If 
$z = r(u)$ is of odd length, then $zz = r(x)$.
If $z = r(u)$ is of even length, then writing $z = ayb$
with $a, b$ single numbers, we have $r(x) = ay(a+b)yb$.
When we speak of a maximal $1$-block in what follows,
we mean one that cannot be extended by additional $1$'s
to the left or right.

It now suffices to prove the following two lemmas:

\begin{lemma}
Let $z \in \naturals^*$, and suppose $|z| \geq 5$ is odd.   Then $zz$ contains either
\begin{itemize}
\item[(a)] two consecutive
terms that are $\geq 2$, or
\item[(b)] six consecutive terms $a1bc1d$ with
    $a \geq c$ and $b \leq d$.
\end{itemize}
\end{lemma}

\begin{proof}
If condition (a) is not satisfied, then $z$ consists of isolated occurrences
of numbers $\geq 2$, separated by blocks of consecutive $1$'s.  We assume this
in what follows.

If $z$ both begins and ends with a number $\geq 2$, then $zz$ satisfies (a).
Thus we may assume that $z$ either begins or ends with $1$ (or both).

Suppose $z$ contains the block $1111$.  Then $zz$ contains the block $b1111c$ for
$b,c \geq 1$, and hence satisfies (b).    Thus we may assume that the maximal $1$-blocks in $z$ are of length $1$, $2$, or $3$.

Suppose all the maximal $1$-blocks of $z$ are of length $1$ or $3$. If $z$
begins with $1$ and ends with $b \geq 2$, then $z$ cannot be of odd length,
and similarly if $z$ ends with $1$ and begins with $b \geq 2$.
So $z$ must begin and end with $1$.
Since $|z| \geq 5$, we know $z$ has a prefix of the form $1c1d$ and a suffix of the
form $a1b1$, where $a,b,c,d \geq 1$.   Hence $zz$ contains the block
$a1b11c1d$.   If $b \leq c$, then the block $a1b11c$ fulfills condition (b);
if $b \geq c$, then the block $b11c1d$ fulfills condition (b).

Thus there must be a maximal $1$-block of length $2$ in $z$.   Then $zz$
contains two blocks, one of the form $a1b11c$ and one of the form $b11c1d$, where $a,d \geq 1$ and $b,c \geq 2$.
 If $b \leq c$, then the block $a1b11c$ fulfills condition (b);
if $b \geq c$, then the block $b11c1d$ fulfills condition (b).
\end{proof}

\begin{lemma}
Let $z \in \naturals^*$, and suppose $|z| \geq 6$ is even, and write
$z = ayb$.   Define $z' = ay(a+b)yb$.  Then $z'$ contains either
\begin{itemize}
\item[(a)] two consecutive
terms that are $\geq 2$, or
\item[(b)] six consecutive terms $a1bc1d$ with
    $a \geq c$ and $b \leq d$, or
\item[(c)] an interior occurrence of $2$.
\end{itemize}
\end{lemma}

\begin{proof}
If condition (a) is not satisfied, then $z$ consists of isolated occurrences
of numbers $\geq 2$, separated by blocks of consecutive $1$'s.  We assume this
in what follows.

If $z$ begins and ends with $1$, then $z'$ has an interior occurrence of $2$,
so (c) is satisfied.  So assume this is not the case.

If $z$ begins $1b$ with $b\geq 2$, then by the previous paragraph it must end
in $c \geq 2$.   Then $z'$ has an occurrence of $(c+1)b$, so (a) is satisfied.
Exactly the same argument works if $z$ ends with $b1$ with $b \geq 2$.
So assume neither of these hold.   

If $z$ has an interior occurrence of $11111$, then $z'$ has an occurrence of
$c11111$, fulfilling condition (b).   If $z$ begins $11111c$ for $c \geq 1$,
it must end with
$d \geq 2$, so $z'$ has an occurrence of $(d+1)1111c$, fulfilling (b).  The
analogous argument holds if $z$ ends $c11111$.   So all maximal $1$-blocks
in $z$ are of length $\leq 4$.

Now we consider the case that $z$ has a maximal block of the form $1111$.
If this occurrence is interior in $z$,
then $z'$ contains the block $b1111c$ for
$b,c \geq 1$, and hence satisfies (b).    
 If $z$ has the prefix $1111$, then $z$ cannot end in $1$ by above.  Hence, since $z$ has even length, it must contain another maximal $1$-block of even length, which must be interior.  Since we have already ruled out the possibility of an interior occurrence of $1111$, it must be an interior $1$-block of size $2$.  But then $z’$ has a block of the form $a1b11c1d$ where $a,d \geq 1$ and $b,c \geq 2$.  As in the previous lemma,
if $b \leq c$, then the block $a1b11c$ fulfills condition (b);
if $b \geq c$, then the block $b11c1d$ fulfills condition (b).
The analogous argument holds if the occurrence of $1111$ is a suffix.
Hence $z$ contains no maximal $1$-block of size $4$.

Thus we may assume that the maximal $1$-blocks in $z$ are of length $1$, $2$, or $3$.

If $z$ has a maximal $1$-block of size $2$, then since $|z|$ is even, it must
have a second maximal $1$-block of size $2$.   Then $z'$ has a factor of the
form $a1b11c1d$ where $a,d \geq 1$ and $b,c \geq 2$, and by the argument above,
this satisfies condition (b).

Hence all the maximal $1$-blocks of $z$ are of size $1$ or $3$.  Hence all the maximal $1$-blocks of $z$ are size $1$ or $3$.  Since $z$ has even length, exactly one of its first and last symbols must be 1.  We know from above that $z$ cannot begin $1b$ or end $b1$ with $b \geq 2$.  So $z$ must either begin $111$ or end $111$.  If $z$ begins $111$, then $z$ ends in $b \geq 2$, and $z'$ contains the block $a1(b+1)11c1d$ for some $a,d \geq 1$ and $b,c \geq 2$.  If $c \leq b+1$, then the block $a1(b+1)11c$ fulfills condition (b); if $b+1 \geq c$, then the block $(b+1)11c1d$
fulfills (b).  If $z$ ends $111$, then an analogous argument holds.
\end{proof}

Applying the two lemmas to the case where $x = u \,
\overline{u}$ and $z = r(u)$ completes the proof of Theorem~\ref{minimal}.
\end{proof}

\section{Enumerating words with only two distinct antisquares}
\label{enum}

In this section we obtain some enumeration results for good words.  Here the notation 
${}^\omega x$ refers to the 
left-infinite word $\cdots xxx$.

\begin{proposition}\label{001011}
If a bi-infinite good word $w$ contains the factor $001011$, then $w={}^\omega0101^\omega$.
\end{proposition}

\begin{proof}
Consider a maximal factor of $w$ of the form $0^k101^k$, with $2\leq k < \infty$.
By maximality and by symmetry, we can assume that $w$ contains $0^k101^k0$.
This factor is not extendable to the right:
\begin{itemize}
 \item $0^k101^k00$ contains the antisquare $1100$ as a suffix.
 \item $0^k101^k01$ is an antisquare.
\end{itemize}
This is a contradiction to $k < \infty$, so $w={}^\omega0101^\omega$.
\end{proof}

\begin{theorem}
There are $\Theta(\psi^n)$ good words of length $n$,
where $\psi \doteq 1.465571231876768$ is the supergolden ratio, root of the equation $X^3=X^2+1$.
\end{theorem}

\begin{proof}
Let $F=\{0011,1100,0110,1001,010101,101010,001011,110100\}$.
By Theorem~\ref{minimal}, $F$ contains the minimal antisquares of order 2 and 3,
and the minimal antisquares of order at least 4 contain $001011$ or $110100$.
Thus, the binary words avoiding $F$ are exactly the good words that avoid
$\{001011,110100\}$.

By Proposition~\ref{001011}, there are not enough good words containing
$001011$ to contribute to the growth rate of good words.
This also holds for good words containing the symmetric factor $110100$.
Thus, good words and binary words avoiding $F$ have the same growth rate.

To enumerate binary words avoiding $F$, we instead enumerate the `Pansiot
codes' of these words.  If $x = x_1x_2\cdots x_n$ is a binary word, then the 
\emph{Pansiot code} of $x$ is the binary word $p_1p_2\cdots p_{n-1}$ such that
for $i=1,\ldots,n-1$
$$
x_{i+1} = \begin{cases}x_i &\text{ if } p_i = 0;\\
\overline{x_i} &\text{ if } p_i = 1.\end{cases}
$$
For example, the binary word $010$ is the Pansiot code for the two binary words
$0011$ and $1100$.

The Pansiot codes of binary words avoiding $F$ are the binary words avoiding\\ $\{ 010,101,11111,01110\}$.
These words consist of blocks of $0$'s of length at least $2$ and blocks
of $1$'s of length $2$ or $4$.
Consider the number $C_n$ of such words ending with $00$.
They are obtained from shorter words by adding a suffix $0$, $1100$, or $111100$.
From the relation $C_n=C_{n-1}+C_{n-4}+C_{n-6}$, the growth rate
is the positive real root of $X^6=X^5+X^2+1$.
Since $X^6-X^5-X^2-1=(X+1)(X^2-X+1)(X^3-X^2-1)$, this is the root $\psi$ of $X^3=X^2+1$.
\end{proof}

Next we show that the threshold exponent at which the number
of good words becomes exponential is $\tfrac{15}4$.
(For overlap-free words, the threshold is $\tfrac73$; see \cite{Karhumaki&Shallit:2004}.)

\begin{theorem}
Let $w$ be any squarefree word over the alphabet $\{0,1,2\}$.  Apply the map $h$ that
sends
\begin{align*}
0 &\rightarrow 010001 \\
1 &\rightarrow 0100010001 \\
2 &\rightarrow 01000100010001 .
\end{align*}
The resulting word is good, and has exponent at most $\tfrac{15}4$,
and it is exactly $\tfrac{15}4$ if $|w| \geq 5$.
\end{theorem}

\begin{proof}
The goodness of $h(w)$ can be seen by inspection.  Regarding the exponent
$\tfrac{15}4$, suppose that $h(w)$ contains a $\tfrac{15}4$
power $zzzz'$, where $z'$ is a prefix of $z$.  Note that in $h(w)$ the
factor $101$ can only occur at the `boundary' between $h(a)$ and $h(b)$,
where $a,b \in \{0,1,2\}$.  So we have two cases:

Case~1: $z$ contains $101$.  Write $z=x101y$.
Then $h(w)$ contains the square $01yx101yx1$, where $01yx1 = h(Z)$ for some
factor $Z$ of $w$.  Then $w$ contains the square $ZZ$, which is a contradiction.

Case~2: $z$ does not contain $101$.  Clearly $z=h(a)$ or $z=h(a)0$ for some $a \in \{0,1,2\}$
is not possible, so $z$ is contained within some $h(a)$, where $a \in \{0,1,2\}$.
The only such $\tfrac{15}4$-power is $h(2)0=(0100)^3 010$, which establishes the claim.
\end{proof}

\begin{corollary}
There are exponentially many length-$n$ $\frac{15}{4}^+$-free good words.
\end{corollary}

To show that there are only polynomially many length-$n$ good
words avoiding $\tfrac{15}4$-powers we need a version of the results in Section~\ref{three}
for finite words rather than bi-infinite words.
Let $g$ and $\varphi$ be defined as in Section~\ref{two} and let $g'$ be the morphism
that maps $0 \mapsto 01$ and $1 \mapsto 00$.

\begin{lemma}\label{g-factorization}
Let $w$ be a $4$-free word of length $\geq 15$ that contains no antisquares other than $01$ and $10$.
Then $w$ can be written as either $w = w_1 g(v) w_2$ or $w = w_1 g'(v) w_2$ for some $v$,
where $|w_1|,|w_2| \leq 5$.
\end{lemma}

\begin{proof}
By a finite search, one verifies that any $w$ satisfying the hypotheses of the lemma
has a prefix of length $\leq 9$ that contains either $0001$ or $0111$.  Suppose it contains
$0111$.  Since $w$ avoids the antisquares $0110$ and $1001$ and the $4$-powers $0000$
and $1111$ we can write $w = w_1 0111z$, where $z$ is a prefix of a word in $\{0001,01,0111\}^*$
and $|w_1| \leq 5$.  We claim that $z$ does not contain $0001$.  Note that $|z| \geq 6$.

If $z$ has $0001$ as a prefix, then $w$ contains the antisquare $111000$.
Suppose $z$ has $01$ as a prefix.  If $z$ has $010001$ as a prefix, then $w$ contains
the antisquare $0111010001$.  If $z$ has $01010$ as a prefix, then $w$ contains
the antisquare $101010$, so necessarily $z$ has $010111$ as a prefix.  Finally, it may be
the case that $z$ has $0111$ as a prefix.  Applying this argument repeatedly to the
suffix of $w$ following this new occurrence of $0111$, until we no longer have such
a suffix of length at least $6$, we see that $z$ does not contain $0001$.  It follows that
$w$ can be written as $w = w_1 x w_2$, where $x \in \{01,0111\}^*$ (and hence $x\in\{01,11\}^*$),
and $|w_2| \leq 5$.  Thus $w = w_1 g(v) w_2$ for some $v$, as required.

A similar argument shows that if $w$ contains $0001$ then $w = w_1 g'(v) w_2$ for some $v$,
where $|w_1| \leq 5$ and  $|w_2| \leq 5$.
\end{proof}

In what follows we will consider words of the form $g(v)$; the analysis for
$g'(v)$ is similar.

\begin{lemma}\label{phi-factorization}
Let $n \geq 1$ and $y = g(\varphi^n(x))$ for some binary word
$x$ with $|x| \geq 5$.  If $y$ is $\tfrac{15}4$-free, then $x$ is $4$-free and
can be written in the form $x = px's$ where
$|p|\leq 2$, $|s| \leq 1$ and $x'$ has no $000$ or $11$.
\end{lemma}

\begin{proof} 
Clearly $x$ is $4$-free.  Suppose $x$ contains an occurrence of $000$ that
is neither a prefix nor a suffix of $x$.  Then $\varphi(x)$ contains
the $4$-power $01(001)^3 0 = (010)^4$ and hence $y$ contains a $4$-power, which is a contradiction.
Suppose $x$ contains an occurrence of $11$ of the form $x=u11v$, where $|u|\geq 2$
and $|v|\geq 1$.  Then $\varphi(x)$ contains $0101$ and hence, extending this occurrence
of $0101$ to the left and right with blocks $001$ and $01$ and avoiding $4$-powers, we see that $\varphi(x)$
contains $001 0101 0$.  Indeed, extending two blocks to the left and one block to the right suffices.
If $n=1$ then $g(00101010)$ contains the $\tfrac{15}4$-power
$(1011)^3 101$, which is a contradiction.  If $n>1$, then $\varphi(00101010)$ contains the $4$-power
$(01001)^4$ and hence $y$ contains a $4$-power, which is a contradiction.
\end{proof}

\begin{lemma}\label{iterate-factorization}
Let $w$ be a $\tfrac{15}4$-free word of length $\geq 33$ that contains no antisquares
other than $01$ and $10$. Then $w$ can be written in the form
$$w = w_1 G(u_1 \varphi(u_2 \cdots \varphi(u_r \varphi(V) v_r) \cdots v_2) v_1) w_2$$
for some $r$, where $G \in \{g,g'\}$, $|w_1|,|w_2| \leq 5$, $|u_i| \leq 4$,
$|v_i| \leq 3$, for $i = 1,\ldots,r$, and $|V|\leq 4$.
\end{lemma}

\begin{proof}
By Lemma~\ref{g-factorization}, we can write $w = w_1 G(v) w_2$, where $|w_1|,|w_2| \leq 5$.
Without loss of generality, suppose $G=g$.  Clearly $v$ must be $4$-free.  Furthermore, $v$ does not contain
$000$ or $11$, since otherwise $g(v)$ would contain either the antisquare $010101$ or the $4$-power $1111$.
Thus $v = u_1\varphi(v')v_1$, where $|u_1|,|v_1| \leq 2$ and $|v'| \geq 5$.  We can then apply
Lemma~\ref{phi-factorization} to $g(\varphi(v'))$ to find that $v'$ is $4$-free and can be written in the form
$v' = px's$, where $|p|\leq 2$, $|s| \leq 1$ and $x'$ has no $000$ or $11$.  Then we can write
$v' = u_2 \varphi(v'') v_2$, where $|u_2|\leq 4$ and $|v_2| \leq 3$,
and repeat the process to obtain the desired decomposition.
\end{proof}

\begin{theorem}
There are polynomially many length-$n$ good words avoiding $\frac{15}4$-powers.
\end{theorem}

\begin{proof}
Let $w$ be such a word of length $n$, where $n \geq 33$.  By Lemma~\ref{iterate-factorization}
$w$ can be written in the form
$$w = w_1 G(u_1 \varphi(u_2 \cdots \varphi(u_r \varphi(V) v_r) \cdots v_2) v_1) w_2$$
for some $r$, where $G \in \{g,g'\}$, $|w_1|,|w_2| \leq 5$, $|u_i| \leq 4$,
$|v_i| \leq 3$, for $i = 1,\ldots,r$, and $|V|\leq 4$.  Suppose $G=g$ and under this assumption
let $A$ (resp.~$B,C,D,E$) be the maximum number of possible choices for
$w_1$ (resp.~$w_2,u_i,v_i,V$).  Then the number of words $w$ is at most
$AB(CD)^rE$.  There is a constant $\rho$ such that $r \leq \rho\log n$,
so the number of words $w$ is at most $ABEn^{\rho\log(CD)}$.
A similar calculation applies when $G=g'$.
\end{proof}

\section{Further work}

In this paper we have studied antisquares.  This situation has an obvious generalization to patterns with morphic and antimorphic permutations, as studied in \cite{Currie&Manea&Nowotka:2015}.   This could be the subject of a future study.

A companion paper to this one is
\cite{Currie&Dvorakova&Ochem&Opocenska&Rampersad&Shallit:2023}, which investigates
complement avoidance in binary words.

\end{document}